\documentclass[11pt]{article}
\usepackage{amssymb,amsfonts,amsmath,latexsym,epsf,tikz,url}

\newtheorem{theorem}{Theorem}[section]

\newtheorem{remark}[theorem]{Remark}
\newtheorem{definition}[theorem]{Definition}

\newcommand{\proof}{\noindent{\bf Proof. }}
\newcommand{\qed}{\hfill $\square$\medskip}

\begin{document}

\title{Introduction to dominated edge chromatic number of  a graph}

\author{Mohammad R. Piri \and
Saeid Alikhani$^{}$\footnote{Corresponding author}}

\date{\today}

\maketitle

\begin{center}
Department of Mathematics, Yazd University, 89195-741, Yazd, Iran\\
{\tt piri4299@gmail.com, alikhani@yazd.ac.ir }
\end{center}

%%%%%%%%%%%%%%ABSTRACT%%%%%%%%%%%%%%%%%%%%%%%%%%%%%%%%%%%%%%%%%%%%%%%%%%%%%%%%%%%%

\begin{abstract}
We introduce and study  the dominated edge coloring of a graph. A dominated edge coloring of a graph $G$ is a proper edge coloring of $G$ such that each color class is dominated by at least one edge of $G$. The minimum number of colors among all dominated edge coloring is called the dominated edge chromatic number, denoted by $\chi_{dom}^{\prime}(G)$.  We obtain some properties of $\chi_{dom}^{\prime}(G)$ and compute it  for specific graphs. Also we examine the effects on $\chi_{dom}^{\prime}(G)$ when $G$ is modified by operations on vertex and edge of $G$. Finally, we consider the $k$-subdivision of $G$ and study the  dominated edge chromatic number of these  kind of graphs.
\end{abstract}

\noindent{\bf Keywords:} dominated edge chromatic number; subdivision; operation; corona. 

\medskip
\noindent{\bf AMS Subj.\ Class.}: 05C25 

%%%%%%%%%%%%%%%%%%%%%%%%%%%%%%%%%%%%%%%%%%%%%%%%%%%%%%%%%%%%%%%%%%%%%%%%%%%%%%%%%
%%%%%%%%%%%%%%%%%%%%%%%%%%%%%%%%%%%%%%%%%%%%%%%%%%%%%%%%%%%%%%%%%%%%%%%%%%%%%%%%%
\section{Introduction and definitions}
%%%%%%%%%%%%%%%%%%%%%%%%%%%%%%%%%%%%%%%%%%%%%%%%%%%%%%%%%%%%%%%%%%%%%%%%%%%%%%%%%
%%%%%%%%%%%%%%%%%%%%%%%%%%%%%%%%%%%%%%%%%%%%%%%%%%%%%%%%%%%%%%%%%%%%%%%%%%%%%%%%%
Let $G=(V, E)$ be a simple graph and $\lambda \in \mathbb{N}$.  If $f$ is a proper coloring of $G$ with the coloring classes $V_1, V_2, ... ,V_{\lambda}$ such that every vertex in $V_i$ has color $i$, then  sometimes proper coloring  write simply $f= (V_1, V_2, ... ,V_{\lambda})$. The chromatic number $\chi (G)$ of $G$ is the minimum of colors needed in a proper coloring of a graph. Similarly, if $f$ is a proper edge coloring of $G$ with the coloring classes $E_1, E_2, ... ,E_{\lambda}$ such that every edge in $E_i$ has color $i$,  write simply $f= (E_1, E_2, ... ,E_{\lambda})$. The edge chromatic number $\chi^{\prime} (G)$ of $G$ is the minimum of colors needed in a proper edge coloring of $G$. 

A dominator coloring of $G$ is a proper coloring of $G$ such that  every vertex of $G$ is adjacent to all vertices of at least one color class. The dominator chromatic number $\chi_{d}(G)$ of $G$ is the minimum number of color classes in a dominator coloring of $G$. The concept of dominator coloring was introduced and studied by Gera, Horton and Rasmussen \cite{e}.
 For a graph $G$ with no isolated vertex, the total dominator coloring is a proper coloring of $G$ in which each vertex of the graph is adjacent to every vertex of some (other) color class. The total dominator chromatic number, abbreviated TD-chromatic number, $\chi_{d}^{t}(G)$ of $G$ is the minimum number of color classes in a TD-coloring of $G$. For more information see \cite{nima1,nima2}. \\
 A set $T$ of vertices is a total dominating set of $G$ if every vertex of $G$ is adjacent to at least one vertex in $T$. The total domination number $\gamma_t(G)$ of $G$ is the minimum number of vertices in a total dominating set of $G$. A set $F$ of edges is an edge dominating set of $G$ if every edge not in $F$ is adjacent to at least one edge in $F$. A set $D$ of edges is a total edge dominating set of $G$ if every edge in $G$ is adjacent to at least one edge in $D$. The edge domination number $\gamma^{\prime}(G)$ and the total edge domination number $\gamma_t^{\prime}(G)$ are the minimum number of edges in an edge dominating set and in a total edge dominating set of $G$, respectively \cite{ccd}.
 
Dominated coloring of a graph is a proper coloring in which each color class is  dominated by a vertex. The least number of colors needed for a dominated coloring of $G$ is called the dominated chromatic number of $G$ and denoted by $\chi_{dom}(G)$ (\cite{Piri,Choopani,cc}).

\medskip 
Motivated by dominated chromatic number of graphs, we consider the proper edge coloring of $G$ and introduce the dominated edge chromatic number of $G$, $\chi_{dom}^{\prime}(G)$, obtain some properties of $\chi_{dom}^{\prime}(G)$ and compute this parameter for specific graphs, in the next section. In Section $3$, we examine the effects on $\chi_{dom}^{\prime}(G)$ when $G$ is modified by operations on vertex and edge of $G$. Finally in Section $4$, we study the dominated edge chromatic number of $k$-subdivision of graphs.

\medskip

\section{Introduction to dominated edge chromatic number}

First we need to introduce some additional but standard notation and definitions.  The maximum degree of a graph
$G$ is denoted by $\Delta(G)$. The open and closed neighborhood of a vertex $x \in V$ are
denoted by $N(x)$ and $N[x]$, respectively. The open neighborhood of an edge $e\in E$ is  $N(e)=\{e'\in E:e' ~\textsl{\rm{ is adjacent to}} ~ e\}$. We denote by $P_n$ the path on $n$ vertices and
by $C_n$ the cycle on $n$ vertices.  The complete graph on $n$ vertices is denoted by $K_n$. The complete bipartite graph with
parts of orders $r$ and $s$ is denoted by $K_{r,s}$ and the star is the complete bipartite graph
$K_{1,k}$  with $k \geq 1$. A bi-star $B_{p,q}$ is a graph formed by two stars $S_p$
and $S_q$ by adding an edge between their center vertices. The join of two graphs $G$ and $H$, denoted by $G+H$, is a graph with vertex set $V(G)\cup V(H)$ and edge set $E(G)\cup E(H)\cup \{uv|u\in V(G), v\in V(H)\}$.

In this section, we state the definition of dominated edge chromatic number and obtain this parameter for some specific graphs.
\begin{definition}
 A dominated edge coloring of graph $G$ is a proper edge coloring of $G$ in which every color class is dominated by at least one edge of $G$. More precisely, a $k$-dominated edge coloring of $G$ is a proper $k$-coloring $\{ C_1, C_2, ..., C_k\}$ of $G$ for every  $i\in \{ 1, 2, ..., k\}$, there exists an edge  $e\in E$ such that $C_i \subseteq N(e
)$. The dominated edge chromatic  number of $G$, $\chi_{dom}^{\prime}(G)$, is the minimum number of color among all dominated edge coloring of $G$.
\end{definition}

Observe that $\chi_{dom}^{\prime}(G)\geq 1$, and $\chi_{dom}^{\prime}(G)=1$ if and only if $G$ is $K_2$.  Also 
$\chi_{dom}^{\prime}(G)\geqslant \gamma_t^{\prime}(G)$, where $\gamma_t^{\prime}(G)$ is total edge domination number of $G$. To see the reason, consider a minimum dominated edge coloring of $G$. 
For constructing a total edge dominating set of $G$, $D_t^{\prime}$,   from each color class we  take one of its dominating edges. The set $D_t^{\prime}$ is an edge dominating set 
with  cardinality $\chi_{dom}^{\prime}(G)$. Moreover, $D_t^{\prime}$ is a total edge dominating set, since each of its edges has a color and is therefore dominated by some other edge of $D_t^{\prime}$.
 
 \begin{remark}\label{zz}
 For every graph $G$, $\chi_{dom}^{\prime}(G)\geqslant \Delta (G)$ and this inequality is sharp. As an example, for the star graph $K_{1,n}$, $\chi_{dom}^{\prime}(K_{1,n})=n$.
 \end{remark} 
 \begin{remark}\label{aa}
 In a dominated edge coloring, every color can be used at most twice. So if $\{C_1, ... , C_t\}$ is color classes of dominated edge coloring, then for every $1\leqslant i \leqslant t$, $\vert C_i\vert =1$ or $\vert C_i\vert=2$. Therefore for every graph $G$ of size  $m$, $\chi_{dom}^{\prime}(G)\geq \lceil \frac{m}{2}\rceil$.
 \end{remark}

 \begin{figure}
 	\begin{center}
 		\includegraphics[width=0.3\textwidth]{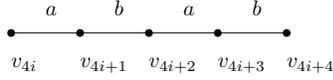}
 		\caption{\label{fv} Dominated edge coloring for $P_5$.}
 	\end{center}
 \end{figure}
 
 \begin{figure}
 	\begin{center}
 		\includegraphics[width=0.9\textwidth]{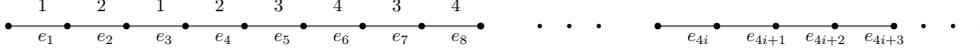}
 		\caption{\label{fvv} Dominated edge coloring for $P_n$.}
 	\end{center}
 \end{figure}
  
 By Remark \ref{aa}, and Figures \ref{fv}, \ref{fvv} and \ref{fw2} we have the following theorem which is about the dominated edge chromatic number of path $P_n$, cycle $C_n$, complete graph $K_n$, complete bipartite graph $K_{m,n}$, wheel graph $W_n$ and friendship graph $F_n:=K_1+nK_2$: 
 \begin{theorem}
 	\begin{enumerate} 
\item[(i)]  \label{path}  For every natural number $n\geq 5$,   
 \[
 	\chi_{dom}^{\prime}(P_n)=\chi_{dom}^{\prime}(C_{n-1})=\left\{
  	\begin{array}{ll}
  	{\displaystyle
  		\dfrac{n-1}{2}}&
  	\quad\mbox{if $n\equiv 0 ~(mod\ 4)$, }\\[15pt]
  	{\displaystyle
  		\lfloor \dfrac{n-1}{2}\rfloor +1} &
  	\quad\mbox{otherwise. }
  	\end{array}
  	\right.	
  	\]
  		\item [(ii)]
  		$\chi_{dom}^{\prime}(K_n)=\lceil \dfrac{n(n-1)}{4}\rceil$.
  		\item[(iii)]  For every $m,n\geq 2$,   $\chi_{dom}^{\prime}(K_{m,n})=\lceil\frac{mn}{2}\rceil$.
  		\item[(iv)] For any $n\geq 3$, $\chi_{dom}^{\prime}(W_n)=n-1$.
  		\item[(v)] For $n\geq 2$,  $\chi_{dom}^{\prime}(F_n)=2n$
  	\end{enumerate} 
 \end{theorem}

 \begin{figure}
 	\begin{center}
 		\includegraphics[width=0.7\textwidth]{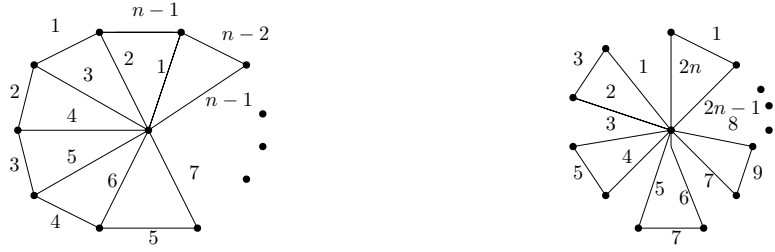}
 		\caption{\label{fw2} Dominated edge coloring of wheel $W_n$ and friendship graph $F_n$, respectively. }
 	\end{center}
 \end{figure}

\begin{theorem}\label{rem}
If $G$ is a connected graph containing $P_7$ as an induced subgraph, then $\chi_{dom}^{\prime}(G)\geq \Delta (G)+2$. More generally, if the graph $P_n$ is an induced subgraph of $G$, then $\chi_{dom}^{\prime}(G)\geq \Delta(G) + \chi_{dom}^{\prime}(P_{n-4})$.
\end{theorem}
\proof
We assign $\Delta (G)$ colors to the edges which are incident to the vertex with maximum degree $\Delta (G)$. Now we consider $P_7$ as induced subgraph of $G$. Since  we need two new colors for each four consecutive edges, so we have $\chi_{dom}^{\prime}(G)\geq \Delta (G)+2$. The proof of $\chi_{dom}^{\prime}(G)\geq \Delta (G)+\chi_{dom}^{\prime}(P_{n-4})$ is similar.
\qed

\begin{remark} 
	The graph $G$ in Figure \ref{fv2} and its coloring shows that the lower bound in Theorem \ref{rem} is sharp.
	\end{remark} 
\begin{figure}
	\begin{center}
		\includegraphics[width=0.4\textwidth]{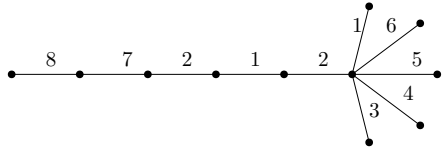}
		\caption{\label{fv2} Graph $G$ with $\chi_{dom}^{\prime}=\Delta (G)+2$.}
	\end{center}
\end{figure}

\begin{figure}
	\begin{center}
		\includegraphics[width=0.5\textwidth]{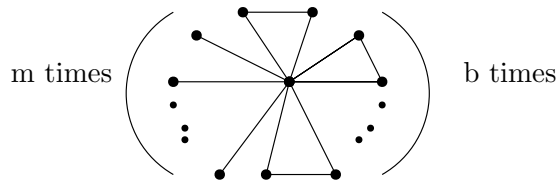}
		\caption{\label{wf1} The graph $G$ in the Theorem \ref{wf2}.}
	\end{center}
\end{figure}
\begin{theorem}
	\begin{enumerate}
		\item [(i)]  \label{wf3}
			If $a$ and $b$ are two integers with $a\geq b \geq 2$ such that $a\geq 2b$, then there exists a graph $G$ with dominated edge chromatic number $\chi_{dom}^{\prime}(G)=a$ and total edge domination number $\gamma_t^{\prime}(G)=b$.
			\item[(ii)]	\label{wf2}
		If $a$ and $b$ are two  integers with $a\geq b\geq 2$, then there exists a graph $G$ with dominated edge chromatic number $\chi_{dom}^{\prime}(G)=a$ and total edge domination number $\gamma_{t}^{\prime}(G)=b$.
	 			\end{enumerate}
\end{theorem}
\proof
	\begin{enumerate}
		\item [(i)] 
	If $a=2b$, then we consider the friendship graph $F_b$. Therefore $\gamma_t^{\prime}(F_b)=b$ and $\chi_{dom}^{\prime}(F_b)=2b$. Now if $a>2b$, then we add $m=a-2b$ pendant edges to center of  the friendship graph $F_b$ and call this new graph $G$  (Figure \ref{wf1}). So $\chi_{dom}^{\prime}(G)=a$ and $\gamma_t^{\prime}(G)=b$.
		\item[(ii)]
		Consider the graph $K_{1,a}$ and connect  $b$ new vertices to $b$ verteices of  degree one (Figure \ref{wf4}). Therefore $\chi_{dom}^{\prime}(G)=a$ and $\gamma_t^{\prime}(G)=b$ \qed
			\end{enumerate}

The corona of $G$ and $H$ is denoted by $G\circ H$, is a graph made by a copy of $G$ (which has $n$ vertices) and $n$ copy of $H$ and joining the $i$-th vertex of $G$ to every vertex in the $i$-th copy of $H$. The following theorem gives the lower and upper bounds for the dominated edge chromatic number of $G\circ H$:
\begin{theorem}\label{corona}
Let $G$ and $H$ be two  graphs of orders  $n$  and $k$, respectively. Then
\begin{enumerate}
\item[i)] for even  $n$, $\lceil \frac{\vert E(G)\vert+\vert E(H) \vert+nk}{2}\rceil\leq \chi_{dom}^{\prime}(G\circ H)\leq \chi_{dom}^{\prime} (G)+n \chi_{dom}^{\prime} (H)+\frac{nk}{2}$,
\item[ii)] for odd  $n$, $\lceil \frac{\vert E(G)\vert+\vert E(H) \vert +nk}{2}\rceil\leq \chi_{dom}^{.\prime}(G\circ H)\leq \chi_{dom}^{\prime} (G)+n \chi_{dom}^{\prime} (H)+\frac{(n-1)k}{2}+k$.
\end{enumerate}
\end{theorem}
\proof
The left inequality follows as Remark \ref{aa}. For the right inequality, we consider a dominated edge coloring for $G$ and $H$ and color $G$ and $n$ copy of $H$ with $\chi_{dom}^{\prime} (G)+n \chi_{dom}^{\prime} (H)$ colors. Then, we color the other edges with $\frac{nk}{2}$ colors, if $n$ is even and with $\frac{(n-1)k}{2}+k$ colors, if $n$ is odd. This is a dominated edge coloring for $G \circ H$ and so we have the result.\qed

\begin{theorem}
Suppose that in a dominated edge coloring (with minimum number of colors) of two graphs $G$ and $H$ of orders $n$ and $k$, respectively, there is no  color class with size one. Then
 \begin{enumerate}
\item[i)] for even  $n$, $\chi_{dom}^{\prime}(G\circ H)=\chi_{dom}^{\prime} (G)+n \chi_{dom}^{\prime} (H)+\frac{nk}{2}$,
\item[ii)] for odd  $n$, $\chi_{dom}^{\prime}(G\circ H)= \chi_{dom}^{\prime} (G)+n \chi_{dom}^{\prime} (H)+\frac{(n-1)k}{2}+k$.
\end{enumerate}
\end{theorem}
\proof
It is similar to the proof of Theorem \ref{corona}.\qed
\begin{remark}
If $t$ is the number of color classes of size one in a dominated edge coloring  of two graphs  $G$ and $H$ of order $n$ and $k$, respectively, then  
$$\chi_{dom}^{\prime}(G\circ H)\leq \chi_{dom}^{\prime} (G)+n \chi_{dom}^{\prime} (H)+\lfloor \frac{n}{2} \rfloor k-t.$$
\end{remark}

 The following theorem gives an upper bound for $\chi_{dom}^{\prime}(G+H)$:
\begin{theorem}
For two connected graph $G$ and $H$,
\begin{equation*}
\chi_{dom}^{\prime}(G+H)\leq \chi_{dom}^{\prime}(G)+\chi_{dom}^{\prime}(H)+\lceil \frac{\vert V(G)\vert \times \vert V(H) \vert}{2}\rceil.
\end{equation*}
\end{theorem}
\proof
We color the graph $G$ with $\chi_{dom}^{\prime}(G)$ and the graph $H$ with $\chi_{dom}^{\prime}(H)$ colors and other edge with $\lceil \frac{\vert V(G)\vert \times \vert V(H) \vert}{2}\rceil$ new colors.
So this is a dominated edge coloring. Note that if $\chi_{dom}^{\prime}(G)=\frac{\vert E(G)\vert}{2}$ and $\chi_{dom}^{\prime}(H)=\frac{\vert E(H)\vert}{2}$, then the inequality is sharp.\qed

\begin{figure}
	\begin{center}
		\includegraphics[width=0.4\textwidth]{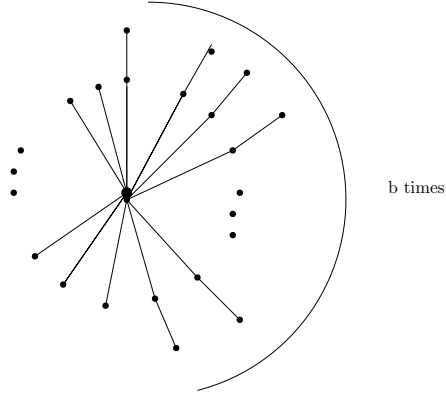}
		\caption{\label{wf4} The graph $G$ in the Theorem \ref{wf3}.}
	\end{center}
\end{figure}

%%%%%%%%%%%%%%%%%%%%%%%%%

\section{Dominated edge chromatic number of some operations on a graph}
 In this section, we examine the effects on $\chi_{dom}^{\prime}(G)$ when $G$ is modified by operations on vertex and edge of $G$.
The graph $G-v$ is a graph that is made by deleting the vertex $v$ and all edges incident
to $v$ from the graph $G$ and the graph $G-e$ is a graph that obtained from $G$ by simply
removing the edge $e$. We present bounds for dominated edge chromatic number of $G -v$
and $G-e$. We begin with $G-v$.

\begin{theorem}\label{delver}
If $G$ is a connected graph and $v\in V(G)$ is not a cut vertex of $G$, then 
\begin{equation*}
\chi_{dom}^{\prime}(G)-deg(v)\leq \chi_{dom}^{\prime}(G-v) \leq \chi_{dom}^{\prime}(G)+deg(v).
\end{equation*}
\end{theorem}
\proof
First we prove the left inequality. We give a dominated edge coloring to $G-v$, add $v$ and all the corresponding edges. Then we assign $deg(v)$ new colors to these edges and do not changes the color of other edges. So this is a dominated edge coloring of $G$ and $\chi_{dom}^{\prime}(G)\leq \chi_{dom}^{\prime}(G-v)+deg(v)$.\\
For the right inequality, first we give a dominated edge coloring to $G$. In this case, since $v$ is not a cut vertex, each edges which is adjacent to an edge with endpoint $v$ has an other adjacent edge too. We change the color of this edge to a new color and do this $deg(v)$ times and do not change the color of the other edges. So this is an edge dominated edge coloring of $G-v$ and $\chi_{dom}^{\prime}(G-v) \leq \chi_{dom}^{\prime}(G)+deg(v)$.
Therefore we have the result.\qed

\begin{remark} 
 The upper bound in Theorem \ref{delver} is sharp. Consider the graph $G$ in Figure \ref{dv1}.
 \end{remark} 
 \begin{figure}
	\begin{center}
		\includegraphics[width=0.2\textwidth]{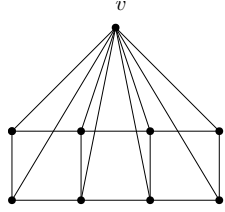}
		\caption{\label{dv1} The graph $G$ in the Theorem \ref{delver}.}
	\end{center}
\end{figure}

 By considering  the graph in Figure \ref{dv2}, we have the following result. 
 \begin{theorem}\label{delvv}
 There is a connected graph $G$ and a vertex $v\in V(G)$ which is not a cut vertex of $G$ such that $\vert\chi_{dom}^{\prime}(G) - \chi_{dom}^{\prime}(G-v)\vert$ can be arbitrarily  large.
 \end{theorem}

\begin{figure}
	\begin{center}
		\includegraphics[width=0.4\textwidth]{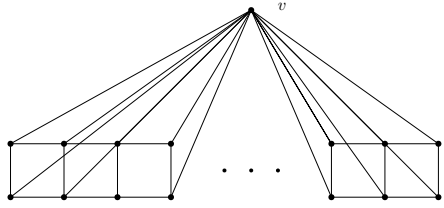}
		\caption{\label{dv2} The graph $G$ in the Theorem \ref{delvv}.}
	\end{center}
\end{figure}
 \begin{theorem}\label{deledge}
 If $G$ is a connected graph, and $e=uv\in E(G)$ is not a bridge of $G$, then 
 \begin{equation*}
\chi_{dom}^{\prime}(G)-1\leq \chi_{dom}^{\prime}(G-e) \leq \chi_{dom}^{\prime}(G)+deg(v)-2.
\end{equation*}
\end{theorem}
\proof
First we prove the left inequality. We give a dominated edge coloring to $G-e$, then we add edge $e$. If we can give one of the previous colors to $e$, then $\chi_{dom}^{\prime}(G-e) =\chi_{dom}^{\prime}(G)$. Otherwise, we assign new color $i$ to edge $e$. So we have a dominated edge coloring for $G$ and $\chi_{dom}^{\prime}(G)\leq \chi_{dom}^{\prime}(G-e)+1$. \\
Now we prove the right inequality. Let $e=uv$ and $deg(u)\geq deg(v)$. If $e$ just dominates its color class, then $\chi_{dom}^{\prime}(G-e) =\chi_{dom}^{\prime}(G)$. If $e$ is the only edge that dominates all adjacent color classes, then by removing $e$, some of these edges will not dominated or previous coloring with removing  $e$ is not a dominated edge coloring. Then we have to add new colors. Since $e$ is not a bridge of $G$, there is a path between $u$ and $v$ other than $e$. In this case, we need to add at least $deg(v)-2$ color. We can use two previous colors, the color of edge $e$ and one of the previous colors (by Remark \ref{aa} and Theorem \ref{path}). So $\chi_{dom}^{\prime}(G-e) \leq \chi_{dom}^{\prime}(G)+deg(v)-2$ and therefore we have the result.
\qed 
\begin{remark}
The bounds in Theorem \ref{deledge} are sharp. For the upper lower consider the graph in Figure \ref{de1}  and for the upper bound consider the graph in Figure\ref{de2}.
\end{remark}
\begin{figure}
	\begin{center}
		\includegraphics[width=0.6\textwidth]{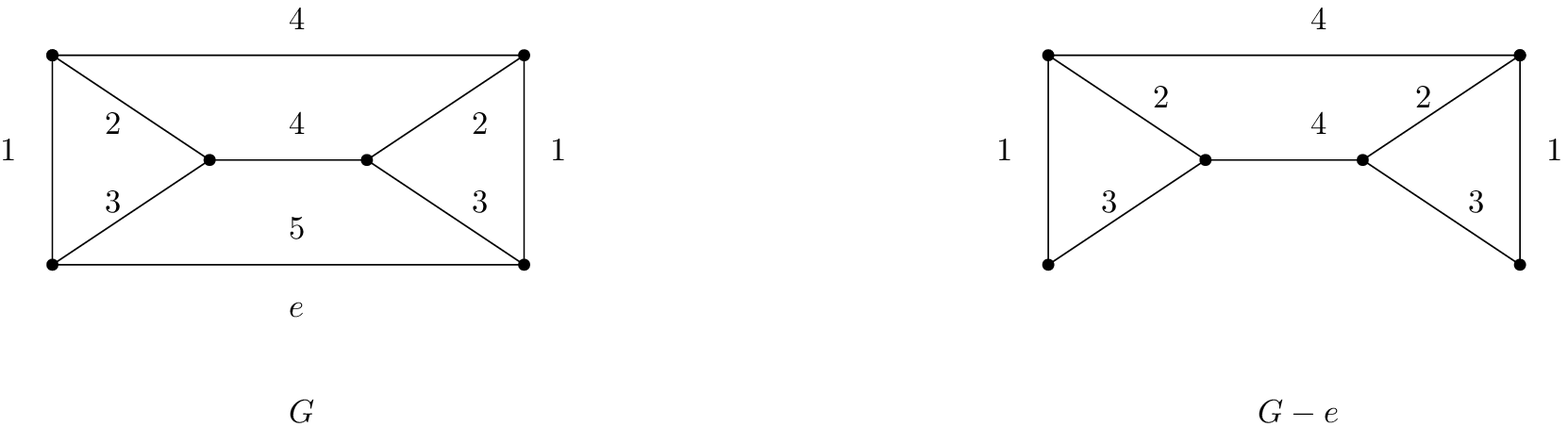}
		\caption{\label{de1} The graph $G$ in the Theorem \ref{deledge}.}
	\end{center}
\end{figure}
\begin{figure}
	\begin{center}
		\includegraphics[width=0.6\textwidth]{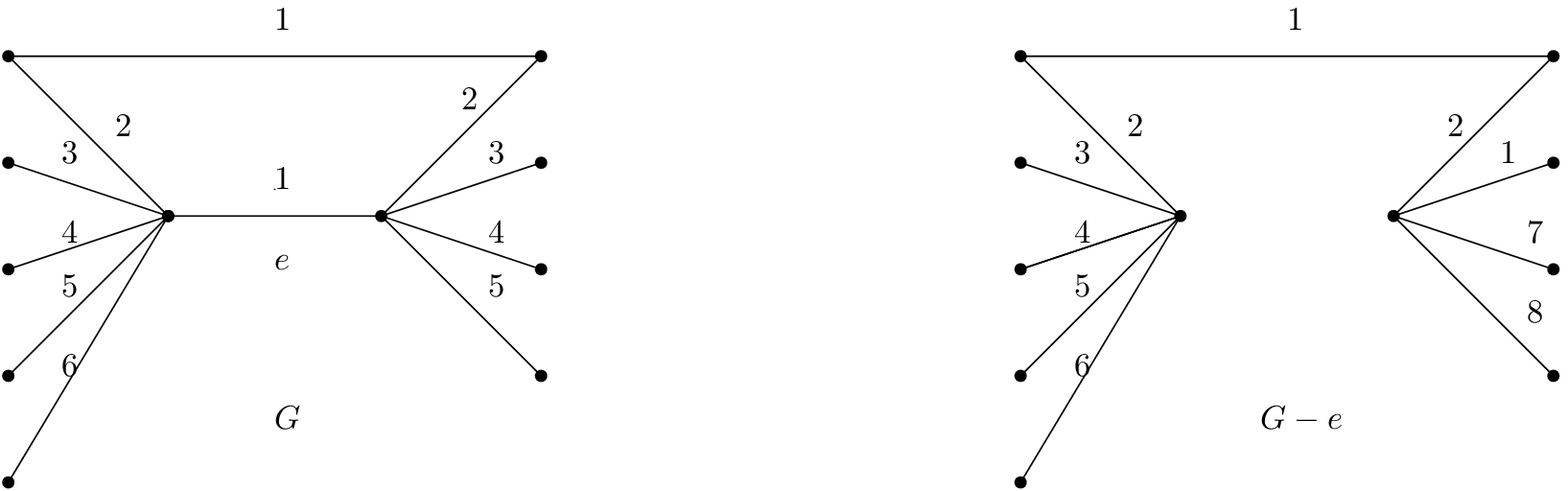}
		\caption{\label{de2} The graph $G$ in the Theorem \ref{deledge}.}
	\end{center}
\end{figure}
\begin{theorem}\label{delv}
 There is a connected graph $G$ and a vertex $v\in V(G)$ which is not a cut vertex of $G$ such that $\vert\chi_{dom}^{\prime}(G) - \chi_{dom}^{\prime}(G-e)\vert$ can be arbitrarily large.
 \end{theorem}
 \proof
 We consider the graph in Figure \ref{de2} with $e=uv$ such that $deg(u)$ and $deg(v)$ are  large enough.
 \qed
 
 \medskip
 
In a graph $G$, contraction of an edge $e$ with endpoints $u, v$ is the replacement of
$u$ and $v$ with a single vertex such that edges incident to the new vertex are the edges
other than $e$ that were incident with $u$ or $v$. The resulting graph $G/e$ has one less edge
than $G$. We denote this graph by $G/e$. We end this section with the following theorem
which gives bounds for $\chi_{dom}^{\prime}(G/e)$.
\begin{theorem}\label{edgecontractin}
Let $G$ be a connected graph and $e=uv\in E(G)$. Then we have:
\begin{equation*}
\chi_{dom}^{\prime}(G)-1\leq \chi_{dom}^{\prime}(G/e)\leq \chi_{dom}^{\prime}(G)+min\{ deg(u), deg(v)\}-1.
\end{equation*}
\end{theorem}
\proof
First we prove the left inequality. We give a dominated edge coloring to $G/e$, add $e$ and assign it a new color, say $i$. This is a dominated edge coloring of $G$. So we have $\chi_{dom}^{\prime}(G)\leq \chi_{dom}^{\prime}(G/e)+1$.\\
For the right inequality, we give a dominated edge coloring to $G$. Suppose that $min \{deg(u), deg(v)\}=deg(u)$. Now we make $G/e$ and change the color of adjacent edge of $e$ with endpoint $u$ to new colors. So we have the result.
\qed

\medskip
\begin{remark}
 The bounds in Theorem \ref{edgecontractin} are sharp. For the upper bound consider the graph bi-star $B_{p,q}$ and for the lower bound consider $C_5$ as $G$. Note that $\chi_{dom}^{\prime}(C_5)=3$ and $\chi_{dom}^{\prime}(C_4)=2$.
 \end{remark}

\section{Dominated edge coloring of $k$-subdivision of a graph}
The $k$-subdivision of $G$, denoted by $G^{\frac{1}{k}}$, is constructed by replacing each edge $v_iv_j$
of $G$ with a path of length $k$., say $P^{\{v_i ,v_j\}}$. These $k$-paths are called superedges, any new
vertex is an internal vertex, and is denoted by $x_l^{\{v_i,v_j\}}$
 if it belongs to the superedge
$P^{\{v_i ,v_j\}}, i < j$ with distance $l$ from the vertex $v_i$, where $l\in \{1, 2,..., k-1\}$. Note that
for $k = 1$, we have $G^{\frac{1}{1}} = G^1 = G$, and if the graph $G$ has $v$ vertices and $e$ edges, then
the graph $G^{\frac{1}{k}}$ has $v + (k - 1)e$ vertices and $ke$ edges. In this section we study dominated edge coloring of $k$-subdivision of a graph (\cite{subnima}). In Particular, we obtain some bounds for $\chi_{dom}^{\prime}(G^{\frac{1}{k}})$ and prove that for any $k\geq2$, $\chi_{dom}^{\prime}(G^{\frac{1}{k}})\leq\chi_{dom}^{\prime}(G^{\frac{1}{k+1}})$.
\begin{theorem}
If $G$ is a graph of size $m$, then $\chi_{dom}^{\prime}(G^{\frac{1}{k}})\geq m$, for $k\geq 3$.
\end{theorem}
\begin{proof}
For $k=3$, in any superedge $P^{\{v,w\}}$ such as $\{v, x_1^{\{v,w\}}, x_2^{\{v,w\}}, w\}$, the edge $x_1^{\{v,w\}}x_2^{\{v,w\}}$ need to use a new color in at least of its adjacent edges and we cannot use this color in any other superedges. So we have the result.
 \qed
 \end{proof}

 \begin{theorem}\label{dfrac}
 If $G$ is a connected graph of size $m$ and $k\geq 2$, then
 \begin{equation*}
\chi_{dom}^{\prime}(P_{k+1})\leq \chi_{dom}^{\prime}(G^{\frac{1}{k}}) \leq m\chi_{dom}^{\prime}(P_{k+1}).
\end{equation*}
 \end{theorem}
\begin{proof}
 First we prove the right inequality. Suppose that $e=uu_1$ be an arbitrary edge of $G$. This edge is replaced with the super edge $P^{\{u,u_1\}}$ in $G^{\frac{1}{k}}$, with vertices $\{u, x_1^{\{u,u_1\}}, ... , x_{k-1}^{\{u,u_1\}}, u_1\}$. We color this superedge with $\chi_{dom}^{\prime}(P_{k+1})$ colors as a dominated edge coloring of $P_{k+1}$. We do this for all superedges. Thus we need at most $m\chi_{dom}^{\prime}(P_{k+1})$ new colors for a dominated edge coloring of $G^{\frac{1}{k}}$.\\
 For the left inequality, if $G$ is a path the result is true. So we suppose that $G$ is a connected graph which is not a path. Let $c^{\prime}$ be a dominated edge coloring of $G^{\frac{1}{k}}$. The restriction of $c^{\prime}$ to edges of $P^{\{u,v\}}$ is a dominated edge coloring and so we have the result.
 \qed
 \end{proof}
 \begin{remark}
 The lower bound of Theorem \ref{dfrac} is sharp for $P_2$ and by the following Theorem we show that the upper bound of this Theorem is sharp for $G=K_{1,n}$ and $k\geq 3$.
 \end{remark}
 \begin{theorem}
 For $n\geq 3$ and $k\geq 3$, $\chi_{dom}^{\prime}(K_{1,n}^{\frac{1}{k}})=n\chi_{dom}^{\prime}(P_{k+1})$.
 \end{theorem}
 \begin{proof}
 Color all path connected to the center vertex with $\chi_{dom}^{\prime}(P_{k+1})$. This is a dominated edge coloring. Because the colors adjacent to the center were used twice, by the Remark \ref{aa}, so we cannot use colors of one path to another path. Therefore we have the result.\qed
 \end{proof}
  \begin{theorem}
 If $G$ is a connected graph and $k \equiv 0 (mod\, 4)$, then $\chi_{dom}^{\prime}(G^{\frac{1}{k}})=m \chi_{dom}^{\prime}(P_{k+1})$.
 \end{theorem}
 \proof
 By Remark \ref{aa} each color can be used at most twice. So we color each superedge with $\chi_{dom}^{\prime}(P_{k+1})$ colors. Therefore $\chi_{dom}^{\prime}(G^{\frac{1}{k}})=m \chi_{dom}^{\prime}(P_{k+1})$.\qed

 \begin{theorem}
 If $G$ is  a graph of size  $m$  and $k \not\equiv 0 (mod\, 4)$ with $k\geq 5$, then 
 \begin{enumerate}
 \item[(i)] If $k\equiv 1 (mod\, 4)$, then $\chi_{dom}^{\prime}(G^{\frac{1}{k}})\geq m\chi_{dom}^{\prime}(P_k)$.
  \item[(ii)] If $k\equiv 2 (mod\, 4)$, then $\chi_{dom}^{\prime}(G^{\frac{1}{k}})\geq m\chi_{dom}^{\prime}(P_{k-1})$.
   \item[(iii)] If $k\equiv  (mod\, 4)$, then $\chi_{dom}^{\prime}(G^{\frac{1}{k}})\geq m\chi_{dom}^{\prime}(P_k)$.
 \end{enumerate}
 \end{theorem}
 \proof
 By Remark \ref{aa}, in each superedge, each color can be used at most twice. So we can use some of colors that used once. If $k \equiv 1 (mod\, 4)$, we need at least $\chi_{dom}^{\prime}(P_k)$ color in each superedges. If $k \equiv 2 (mod\, 4)$, we need at least $\chi_{dom}^{\prime}(P_{k-1})$ color in each superedges and if $k \equiv 3 (mod\, 4)$, we need at least $\chi_{dom}^{\prime}(P_k)$ color in each superedges. Therefore, we have the result.\qed
 
 \begin{theorem}\label{zzz}
 For any $k\geq 3$, $\chi_{dom}^{\prime}(G^{\frac{1}{k}})\leq \chi_{dom}^{\prime}(G^{\frac{1}{k+1}})$.
 \end{theorem}
 \begin{proof}
 First we give a dominated edge coloring of $G^{\frac{1}{k}}$. Let $P^{\{u,v\}}$ be an arbitrary superedge of $G^{\frac{1}{k+1}}$ with vertex set $\{ v, x_1^{\{v,w\}}, x_2^{\{v,w\}}, ... , x_k^{\{v,w\}}, w\}$. There exists an edge $u\in \{x_1^{\{v,w\}}x_2^{\{v,w\}}, ... , x_{k-1}^{\{v,w\}}, x_k^{\{v,w\}}\}$. Consider the graph in Figure \ref{paa}. We have the following cases:
 \begin{figure}
	\begin{center}
		\includegraphics[width=0.35\textwidth]{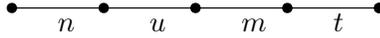}
		\caption{\label{paa} A part of a superedge in the proof of Theorem \ref{zzz}.}
	\end{center}
\end{figure}
\begin{enumerate}
\item[Case 1)] Suppose that the edge $u$ has color $i$ and the edge $n$ has color $j$ and the edge $m$ has color $t$. In this case, we make $G/u$ and do not change the color af any edge. So without adding a new color we have a dominated edge coloring for this new graph.
\item[Case 2)] Suppose that the edge $u$ has color $i$ and the edge $n$ has color $j$ and the edge $m$ has color $j$. In this case, we make $G/u$ and change one of the edge with color $j$ to color $i$ such that two edge with color $i$ not adjacent. So without adding a new color, we have a dominated edge coloring for this new graph.
\end{enumerate}
Now we do the same algorithm for all superedges. So we have a dominated edge 
coloring again.\qed
 \end{proof}
 \begin{theorem}\label{zdf}
 For any graph $G$, $\chi_{dom}^{\prime}(G^{\frac{1}{2}})\leq \chi_{dom}^{\prime}(G^{\frac{1}{3}})$.
 \end{theorem}
 \begin{proof}
 First we give a dominated edge coloring to the edges of $G^{\frac{1}{3}}$. Let $P^{\{w,z\}}$ be an arbitrary superedge of $G^{\frac{1}{3}}$ with edge set $\{s,v,u\}$ (see Figure \ref{df1}) and suppose that edge $v$ has the color $i$.
 \begin{figure}
	\begin{center}
		\includegraphics[width=0.4\textwidth]{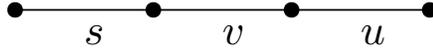}
		\caption{\label{df1} A part of a superedge in the proof of Theorem \ref{zdf}.}
	\end{center}
\end{figure}
We have the following cases:
\begin{enumerate}
\item[Case 1)]
The edge $u$ has the color $j$ and the edge $s$ has the color $t$. In this case, we make $G/v$ and don`t change the color of any edges. So we have a dominated edge coloring for this new graph.
\item[Case 2)] The edge $u$ and $v$ have the color $j$. In this case, we make $G/v$ and change one of the edge with color $j$ to color $i$ such that two edge with color $i$ not adjacent. So without adding a new color we have a dominated edge coloring for this new graph.
\end{enumerate}
Now we do the same algorithm for all superedge. Therefore we have a dominated edge coloring again.\qed
 \end{proof}
 

\begin{thebibliography}{99}



	
	\bibitem{ars} S. Alikhani,  Emeric Deutsch, {\it More on domination polynomial and domination root}, Ars Combin. 134, (2017) 215-232.
	
		
		\bibitem{Piri} S. Alikhani, Mohammad R. Piri, {\it Dominated chromatic number of some operations on a graph}, Available at \texttt{https://arxiv.org/abs/1912.00016}. 
			
			\bibitem{Choopani}	F. Choopani, A. Jafarzadeh, D.A. Mojdeh, {\it On dominated coloring of graphs and some Nardhaus-Gaddum-type relations},  Turkish J. Math. 42 (2018) 2148-2156.


\bibitem{e}
 R. Gera, S. Horton, C. Ramussen, {\it Dominator colrings and safe clique partitios}, Conress. Num., 181  (2006) 19-32.


 
 
 \bibitem{nima1}
 N. Ghanbari, S. Alikhani,  {\it More on the total dominator chromatic number of a graph},  J. Inform. Optimiz. Sci.,  40 (2019), no. 1, 157--169.
 
 \bibitem{nima2} 
 N. Ghanbari, S. Alikhani,  {\it Total dominator chromatic number of some operations on a graph}, Bull. Comp. Appl. Math., 6 (2018), no. 2, 9-20.
 
 \bibitem{subnima} N. Ghanbari, S. Alikhani,  {\it Introduction to total dominator edge chromatic number}, Available at \texttt{https://arxiv.org/abs/1801.08871}.
 
  \bibitem{ccd}
  V. R. Kulli, D. K. Patwari,  {\it  On the total edge domination number of graph}, In A. M. Mathi, ed. Proc. of the symp. on Graph Theory and Combinatorics, Kochi Centre Math. Sci, Trivandrum, Series Publication  21 (1991) 75-81.
 
 \bibitem{cc}
 H.B. Merouane, M. Chellali, M. Haddad, H. Kheddouci, {\it  Dominated coloring of graphs}, Graphs Combin. 31 (2015) 713-727.
 

  
 
  \bibitem{walsh} M. Walsh, {\it The hub number of a graph}, Int. J. Math. Comput. Sci., 1
(2006) 117-124.
  

\end{thebibliography}
\end{document}